\documentclass[aop,MSNbibl,seceqn,citesort,dvips]{arximspdf}
\usepackage{graphicx}

\doi{10.1214/10-AOP641}
\volume{40}
\issue{3}
\pubyear{2012}
\firstpage{893}
\lastpage{920}

\makeatletter

\newcommand{\R}{{\mathbb{R}}}
\newcommand{\Z}{{\mathbb{Z}}}
\newcommand{\PP}{\mathbb{P}}
\newcommand{\TT}{\hat{\tau}}
\newcommand{\E}{\mathbb{E}}
\renewcommand{\P}{\mathbb{P}}

\newtheorem{theorem}{Theorem}
\newtheorem{corollary}{Corollary}
\newtheorem{lemma}{Lemma}

\newproclaim{remark}{Remark}

\makeatother

\begin{document}
\begin{frontmatter}

\title{Limit theorems for 2D invasion percolation}
\runtitle{2D invasion percolation}

\begin{aug}
\author[A]{\fnms{Michael} \snm{Damron}\corref{}\thanksref{t1}\ead[label=e1]{mdamron@princeton.edu}} and
\author[B]{\fnms{Art\"{e}m} \snm{Sapozhnikov}\thanksref{t2}\ead[label=e2]{artem.sapozhnikov@math.ethz.ch}}
\runauthor{M. Damron and A. Sapozhnikov}
\affiliation{Princeton University and ETH Z\"{u}rich}
\address[A]{Princeton University\\
Fine Hall, Washington Rd.\\
Princeton, New Jersey 08544\\
USA\\
\printead{e1}}
\address[B]{ETH Z\"{u}rich\\
R\"{a}mistrasse 101\\
8092 Z\"{u}rich\\
Switzerland\\
\printead{e2}}
\end{aug}

\thankstext{t1}{Supported by an NSF postdoctoral fellowship.}
\thankstext{t2}{Supported in part by the Excellence Fund Grant of TU/e
of Remco van der Hofstad.}

\received{\smonth{9} \syear{2010}}

%
\begin{abstract}
We prove limit theorems and variance estimates for quantities related
to ponds and outlets for 2D invasion percolation. We first exhibit
several properties of a sequence $({\mathbf O}(n))$ of outlet
variables, the $n$th of which gives the number of outlets in the box
centered at the origin of side length $2^n$. The most important of
these properties describes the sequence's renewal structure and
exponentially fast mixing behavior. We use these to prove a central
limit theorem and strong law of large numbers for $({\mathbf O}(n))$.
We then show consequences of these limit theorems for the pond radii
and outlet weights.
\end{abstract}

%

\end{frontmatter}

\section{Introduction}\label{sec:introduction}

\subsection{The model}\label{subsec:model}

Invasion percolation is a stochastic growth model both introduced and
numerically studied independently by \cite{Chandler} and \cite{Lenormand}.
Let $G = (V,E)$ be an infinite connected graph in which a distinguished
vertex, the origin, is chosen.
Let $(\tau_e)_{e\in E}$ be independent random variables, uniformly
distributed on $[0,1]$. The
\textit{invasion percolation cluster} (IPC) of the origin on $G$ is
defined as
the limit of an increasing sequence $(G_n)$ of connected subgraphs of
$G$ as follows.
For an arbitrary subgraph $G' = (V',E')$ of $G$, we define the outer
edge boundary of $G'$ as
\[
\Delta G'=\{e=\langle x,y\rangle\in E\dvtx e\notin
E'\mbox{, but }x\in V'\mbox{ or }y\in V'\}.
\]
We define $G_0$ to be the origin.
Once the graph $G_i=(V_i,E_i)$ is defined, we select the edge $e_{i+1}$
that minimizes $\tau$ on $\Delta G_i$.
We take $E_{i+1}=E_i\cup\{e_{i+1}\}$ and let $G_{i+1}$ be the graph
induced by the edge set $E_{i+1}$.
The graph $G_i$ is called the \textit{invaded region} at time
$i$.\vadjust{\goodbreak}
Let $E_\infty= \bigcup_{i=0}^\infty E_i$ and $V_\infty= \bigcup_{i=0}^\infty V_i$. Finally, define the IPC
\[
\mathcal{S} = (V_\infty, E_\infty).
\]

We study invasion percolation on two-dimensional lattices; however, for
simplicity \textit{we restrict ourselves hereafter to the square lattice}
${\Z}^2$ and denote by ${\E}^2$ the set of nearest-neighbour edges.
The results of this paper still hold for lattices which are invariant
under reflection in one of the coordinate axes and under rotation
around the origin by some angle. In particular, this includes the
triangular and honeycomb lattices.

We define Bernoulli percolation using the random variables $\tau_e$ to
make a coupling with the invasion immediate.
For any $p\in[0,1]$ we say that an edge $e\in{\mathbb E}^2$ is
$p$-\textit{open} if $\tau_e< p$ and $p$-\textit{closed} otherwise.
It is obvious that the resulting random graph of $p$-open edges has the
same distribution as
the one obtained by declaring each edge of ${\mathbb E}^2$ open with
probability $p$ and closed with probability $1-p$,
independently of the state of all other edges.
The percolation probability $\theta(p)$ is the probability that the
origin is in the infinite cluster of $p$-open edges.
There is a critical probability $p_c = \inf\{p\dvtx \theta(p)>0\}\in(0,1)$.
For general background on Bernoulli percolation we refer the reader to
\cite{Grimmett}.

In \cite{Newman}, it is shown that, for any $p>p_c$, the invasion on
$(\Z^d,\E^d)$ intersects the infinite $p$-open cluster with
probability one.
In the case $d=2$ this immediately follows from the
Russo--Seymour--Welsh theorem (see Section 11.7 in \cite{Grimmett}).
This result has been extended to much more general graphs in \cite{HPS}.
Furthermore, the definition of the invasion mechanism implies that if
the invasion reaches the $p$-open infinite cluster for some $p$,
it will never leave this cluster.
Combining these facts yields that if $e_i$ is the edge added at step~%
$i$, then $\limsup_{i\to\infty}\tau_{e_i}=p_c$.
It is well known that for Bernoulli percolation on $(\Z^2,\E^2)$, the
percolation probability at $p_c$ is $0$.
This implies that, for infinitely many values of $i$, the weight $\tau
_{e_i}$ satisfies $\tau_{e_i}>p_c$.
The last two results give that $\TT_1=\max\{\tau_e\dvtx e\in E_\infty\}
$ exists and is greater than~$p_c$.
The above maximum is attained at an edge which we shall call $\hat e_1$.
Suppose that~$\hat{e}_1$ is invaded at step $i_1$, that is, $\hat e_1
= e_{i_1}$.
Following the terminology of~\cite{Newman-Stein}, we call the graph
$G_{i_1-1}$ the \textit{first pond} of the invasion, denoting it by the
symbol~$\hat V_1$, and we call the edge $\hat{e}_1$ the \textit{first outlet}.
The second pond of the invasion is defined similarly.
Note that a simple extension of the above argument implies that $\TT
_2=\max\{\tau_{e_i}\dvtx e_i\in E_\infty,i>i_1\}$ exists and is greater
than~$p_c$.
If we assume that $\TT_2$ is taken on the edge $\hat{e}_2$ at step $i_2$,
we call the graph $G_{i_2-1}\setminus G_{i_1-1}$ the \textit{second pond}
of the invasion, and we denote it $\hat V_2$.
The edge $\hat e_2$ is called the \textit{second outlet}.
The further ponds $\hat V_k$ and outlets $\hat e_k$ are defined analogously.
For a hydrological interpretation of the ponds we refer the reader to
\cite{BJV}.

In this paper, we consider a sequence of outlet variables introduced in~\cite{DS}.
We continue the analysis from that paper, in which almost
sure bounds were shown for the sequence's growth rate. Here, we prove
limit theorems for the sequence and, as a consequence, we obtain
variance estimates for the sequence $(\hat\tau_k)$ of outlet weights
and for the sequence of pond radii. The current results were inspired
by limit theorems for critical percolation obtained by Kesten and Zhang
in \cite{kesten-zhang} and later by Zhang in \cite{Zhang}. In those
papers, the authors prove central limit theorems for (a) the maximal
number of disjoint open circuits around the origin in the box of size
$n$ centered at the origin in critical percolation in two dimensions
and (b) the number of open clusters in the same box in any dimension in
percolation with parameter $p \in[0,1]$. The martingale methods they
use apply to some degree for our questions of invasion percolation, but
our techniques, based on mixing properties and moment bounds from \cite
{DS}, seem to reveal more of the underlying structure of the process.

The mixing properties mentioned above are consequences of a more
general renewal mechanism that lies inside the invasion process on $\Z
^2$. In Section~\ref{sec:outlets}, we show that for any $m,k \geq1$,
the invaded regions at distances $2^m$ and~$2^{m+k}$ from the origin
are equal to two statistically independent sets except on an event
whose probability decays exponentially in $k$. Roughly speaking, this
means that the invasion has a very weak dependence structure when
viewed on exponential length scales.

Last we would like to mention that limit theorems similar to ones we
establish in this paper were shown by Goodman \cite{Goodman} for
invasion percolation on the regular tree. Those results were also
inspiration for the current work. Goodman showed, for example, that the
sizes of the ponds grow exponentially, with laws of large numbers,
central limit theorems and large deviation results. His analysis is
based on representing $\mathcal{S}$ in terms of the outlets weights
$\hat\tau_n$, as in \cite{AGdHS}.

\subsection{Notation}\label{subsec:Notation}

In this section we collect most of the notation and the definitions
used in the paper.

For $a\in\R$, we write $|a|$ for the absolute value of $a$, and,
for a site $x = (x_1$, $x_2)\in\Z^2$, we write $|x|$ for $\max(|x_1|,|x_2|)$.
For $n>0$ and $x\in\Z^2$,
let $B(x,n) = \{y\in\Z^2\dvtx |y-x|\leq n\}$ and $\partial B(x,n) = \{
y\in\Z^2\dvtx |y-x|=n\}$.
We write $B(n)$ for $B(0,n)$ and $\partial B(n)$ for $\partial B(0,n)$.
For $m<n$ and $x\in\Z^2$, we define the annulus $\operatorname{Ann}(x;m,n) =
B(x,n)\setminus B(x,m)$.
We write $\operatorname{Ann}(m,n)$ for $\operatorname{Ann}(0;m,n)$.

We consider the square lattice $(\Z^2,{\mathbb E}^2)$, where ${\mathbb
E}^2 = \{\langle x,y \rangle\in\Z^2\times\Z^2\dvtx |x-y|=1\}$.
Let $(\Z^2)^* = (1/2,1/2) + \Z^2$ and $({\mathbb E}^2)^* = (1/2,1/2)
+ {\mathbb E}^2$ be the vertices and the edges of the dual lattice.
For $x\in\Z^2$, we write $x^*$ for $x + (1/2,1/2)$.
For an edge $e\in{\mathbb E}^2$ we denote its endpoints (left,
resp., right or bottom, resp., top) by $e_x, e_y\in\Z^2$.
The edge $e^* = \langle e_x + (1/2,1/2),e_y-(1/2,1/2) \rangle$ is
called the \textit{dual edge} to $e$.
Its endpoints (bottom, resp., top or left, resp., right) are
denoted by $e_x^*$ and $e_y^*$.
Note that $e_x^*$ and $e_y^*$ are not the same as $(e_x)^*$ and $(e_y)^*$.
For a subset ${\mathcal K}\subset\Z^2$, let ${\mathcal K}^* =
(1/2,1/2) + {\mathcal K}$.
We say that an edge $e\in{\mathbb E}^2$ is in ${\mathcal K}\subset\Z
^2$ if both its endpoints are in ${\mathcal K}$. For any graph
$\mathcal{G}$ we write $|\mathcal{G}|$ for the number of vertices
in~$\mathcal{G}$.

Let $(\tau_e)_{e\in{\mathbb E}^2}$ be independent random variables,
uniformly distributed on $[0,1]$, indexed by edges.
We call $\tau_e$ the \textit{weight} of an edge $e$.
We define the weight of an edge $e^*$ as $\tau_{e^*} = \tau_e$.
We denote the underlying probability measure by ${\mathbb P}$ and the
space of configurations by $([0,1]^{{\mathbb E}^2},{\mathcal F})$,
where ${\mathcal F}$ is the natural $\sigma$-field on $[0,1]^{{\mathbb E}^2}$.
We say that an edge $e$ is $p$-\textit{open} if $\tau_e <p$ and $p$-\textit{closed} if $\tau_e > p$.
An edge $e^*$ is $p$-open if~$e$ is $p$-open, and it is $p$-closed if~$e$ is $p$-closed.
The event that two sets of sites ${\mathcal K}_1,{\mathcal K}_2\subset
\Z^2$ are connected by a~$p$-open path is denoted by
${\mathcal K}_1 \stackrel{p}\longleftrightarrow{\mathcal K}_2$.

For any $k \geq1$, let $\hat R_k$ be the radius of the union of the
first $k$ ponds. In other words,
\[
\hat R_k = \max\Biggl\{|x|\dvtx  x \in\bigcup_{j=1}^k \hat V_k\Biggr\}.
\]
For two functions $g$ and $h$ from a set ${\mathcal X}$ to ${\mathbb R}$,
we write $g(z) \asymp h(z)$ to indicate that $g(z)/h(z)$ is bounded
away from $0$ and $\infty$, uniformly in $z\in{\mathcal X}$. We will
also use the standard notation $g(z) = O(h(z))$ if $g(z)/h(z)$ is
bounded away from $\infty$ uniformly in $z\in{\mathcal X}$, and
$g(z)=o(h(z))$ if for each $\varepsilon>0$, $|g(z)/h(z)|>\varepsilon$
for only a finite number of values of $z\in{\mathcal X}$. For any
event $A$, we write $I(A)$ for the indicator function of $A$. For any
sequence of random variables $(X_i)$ and any $k \geq0$, we say that
the sequence is $k$-\textit{dependent} if for every $m \geq1$, the set of
variables $\{X_1, \ldots, X_m\}$ is independent of the set of
variables $\{X_{m+k+1}, \ldots\}$. Similarly we say that the sequence
of events $(A_i)$ is $m$-dependent if the sequence of variables
$(I[A_i])$ is. Throughout this paper we write $\log$ for $\log_2$.
All the constants $(C_i)$ in the proofs are strictly positive and finite.
Their exact values may be different from proof to proof.\vspace*{-2pt}

\subsection{Main results}\label{subsec:mainresults}\vspace*{-2pt}

\subsubsection{The CLT for outlets}

Let $O_k$ be the number of outlets in the annulus $\operatorname{Ann}(2^{k-1},2^{k})$
and $a_k = \E O_k$. Let ${\mathbf O}(n) = \sum_{k=1}^n O_k, a(n) =
\E{\mathbf O}(n)$ and $b(n)^2 = \operatorname{Var} {\mathbf O}(n)$.
\begin{theorem}\label{thm:VarOn}
There exist positive and finite constants $c_1$ and $c_2$ such that for
all $i$, $a_i\in[c_1,c_2]$, and the variance of ${\mathbf O}(n)$ satisfies
\[
b(n)^2 \asymp n.
\]
\end{theorem}

Write $N(0,1)$ for the distribution of a standard normal random
variable, and let $\Rightarrow$ denote convergence in distribution.
\begin{theorem}\label{thm:cltokthm}
The sequence $({\mathbf O}(n))$ satisfies a CLT, that is,
%
%
\begin{equation}\label{eq:CLT}
\frac{{\mathbf O}(n)-a(n)}{b(n)} \Rightarrow N(0,1).\vadjust{\goodbreak}
\end{equation}
Furthermore, if $r > 1/2$, then the following convergence is almost sure:
%
%
\begin{equation}\label{eq:SLLN}
\frac{{\mathbf O}(n) - a(n)}{n^r} \to0.
\end{equation}
\end{theorem}

\subsubsection{Consequences of the CLT for outlets}
As discussed in Section \ref{subsec:model}, a~main intention of this
paper is to study the asymptotic behavior of the sequences $(\hat R_n)$
(toward infinity) and $(\hat\tau_n)$ (toward $p_c$). In \cite{DS} it
was proved that these sequences obey the following almost sure bounds.
There exist constants $C_1>0$ and $C_2<\infty$ such that with
probability one
\[
C_1n \leq\log\hat R_n \leq C_2n \quad\mbox{and}\quad C_1n \leq-\log
(\hat\tau_n - p_c) \leq C_2n
\]
for all large $n$. Motivated by these results, we want study whether or
not these sequences converge, after properly shifting and normalizing.
Further, we would like know information about rates of convergence. It
turns out that from the point of view of these questions, the sequences
are closely related to certain sequences $(Q_n)$ and $(T_n)$, which we
now define.

Let
\[
Q_n = \min\{k\dvtx {\mathbf O}(k)\geq n\} \quad\mbox{and}\quad T_n = \min\{
k\dvtx a(k)\geq n\}.
\]
Note that ${\mathbf O}(Q_n-1)<n\leq{\mathbf O}(Q_n)$ and
$a(T_n-1)<n\leq a(T_n)$. We define a~sequence of random variables
$(a(Q_n))$, where $a(Q_n)$ equals $a(k)$ if and only if $Q_n=k$. By
this definition, $a(Q_n)$ takes values in the set $\{a(k)\dvtx k \geq1\}$ with
\[
\P\bigl(a(Q_n) = a(k)\bigr) = \P(Q_n = k).
\]
The CLT for outlets allows us to study the sequence $(a(Q_n))$. Let
$\sigma_n^2 = \operatorname{Var}{\mathbf O}(T_n)$.
\begin{theorem}\label{thm:inverseCLT}
\[
\frac{a(Q_n) -n}{\sigma_n} \Rightarrow N(0,1).
\]
[Or, equivalently, $(a(Q_n)-a(T_n))/\sigma_n \Rightarrow N(0,1)$.]
Moreover,
\[
\E\biggl( \frac{a(Q_n)-n}{\sigma_n} \biggr) ^2 \to1\qquad \mbox{as
} n \to\infty.
\]
\end{theorem}
\begin{remark}
We would like to use Theorem \ref{thm:inverseCLT} to deduce CLTs for
the sequences $(\log\hat R_n)$ and $(\log(\hat\tau_n-p_c))$, both
of which are proved in the case of the regular tree in \cite{Goodman}.
It is not difficult to prove these results if one knows that
$(Q_n-T_n)/ \delta_n$ converges in distribution to a variable with the
standard normal distribution for some sequence $\delta_n$.
Unfortunately, Theorem \ref{thm:inverseCLT} does not appear to be
strong enough to show this. One possible approach to deduce a CLT
for
$(Q_n)$ from Theorem \ref{thm:inverseCLT} is to demonstrate that the
sequence~$(a_n)$ does not fluctuate too quickly as $n \to\infty$. For
instance one could try to\vadjust{\goodbreak} prove that there exists $a\in\R$ such that
for every sequence $(k_n)$ of natural numbers,
\[
(1/n)\sum_{k_n+1}^{k_n+n} a_i \to a \qquad\mbox{as } n \to\infty.
\]
\end{remark}

Although we are not able to prove CLTs for $(\log\hat R_n)$ and $(\log
(\hat\tau_n-p_c))$, we show in the next corollaries that the
fluctuations are of the correct order of magnitude.
\begin{corollary}\label{cor:inversevar}
\[
{\mathbb E}(Q_n-T_n)^2\asymp n,\qquad{\mathbb E}(Q_n-{\mathbb
E}Q_n)^2\asymp n.
\]
\end{corollary}
\begin{corollary}\label{cor:radii}
\[
{\mathbb E}(\log\hat R_n - T_n)^2 \asymp n,\qquad\E(\log\hat R_n - \E
\log\hat R_n)^2 \asymp n.
\]
\end{corollary}

In the statements of the next two corollaries, we use the sequence
$(p_n)$, defined in Section \ref{sec:corlength}.
\begin{corollary}\label{cor:weights}
\[
{\mathbb E}\biggl(\log\frac{\hat\tau_n-p_c}{p_{2^{T_n}}-p_c}
\biggr)^2 \asymp n,\qquad\E\bigl(\log(\hat\tau_n - p_c) - \E\log(\hat\tau_n
- p_c)\bigr)^2 \asymp n.
\]
\end{corollary}

Last we show that the sequences $(Q_n)$, $(\log\hat R_n)$ and $(\log
(\hat\tau_n - p_c))$ satisfy laws of large numbers.
\begin{corollary}\label{cor:SLLN2}
For any $r > 1/2$, each of the following sequences converges to 0
almost surely:
\[
\biggl(\frac{Q_n-T_n}{n^r}\biggr),\qquad\biggl(\frac{\log\hat R_n -
T_n}{n^r}\biggr),\qquad\biggl(\frac{1}{n^r} \log\frac{\hat\tau
_n-p_c}{p_{2^{T_n}}-p_c}\biggr).
\]
\end{corollary}

\subsection{Structure of the paper}\label{subsec:structure}
In Section \ref{sec:corlength} we recall the definition of the
correlation length, which is vital to all of our proofs. In
Section \ref{sec:outlets} we describe and prove several properties of
the outlet variables $(O_k)$ that will be used in the proofs of
Theorems \ref{thm:VarOn} and \ref{thm:cltokthm} in Section \ref
{sec:CLT}. In Section \ref{sec:further}, we prove consequences of the
CLT: Theorem \ref{thm:inverseCLT} and Corollaries \ref
{cor:inversevar}--\ref{cor:SLLN2}.

\section{Correlation length}\label{sec:corlength}

\subsection{Definition of correlation length}

For $m,n$ positive integers and $p\in(p_c,1]$ let
\[
\sigma(n,m,p) = \PP(\mbox{there is a $p$-open horizontal crossing of
}[0,n]\times[0,m]).
\]
Given $\varepsilon>0$, we define
%
%
\begin{equation}
L(p,\varepsilon)=\min\{n\dvtx  \sigma(n,n,p)\geq1-\varepsilon\}.\vadjust{\goodbreak}
\end{equation}
$L(p,\varepsilon)$ is called the \textit{finite-size scaling correlation length}
and it is known that~$L(p,\varepsilon)$ scales like the usual
correlation length (see \cite{kesten}).
It was also shown in \cite{kesten} that the scaling of~$L(p,\varepsilon)$ is
independent of~$\varepsilon$ given that it is
small enough,
that is, there exists $\varepsilon_0>0$ such that for all
$0<\varepsilon_1,\varepsilon_2\leq\varepsilon_0$
we have $L(p,\varepsilon_1)\asymp L(p,\varepsilon_2)$. (Here,
$\varepsilon_1$ and $\varepsilon_2$ are fixed numbers that do not
depend on $p$.) For simplicity we will write $L(p)=L(p,\varepsilon_0)$
in the entire paper.
We also define
\[
p_n=\sup\{p\dvtx  L(p)>n\}.
\]
It is easy to see that $L(p)\to\infty$ as $p\to p_c$ and $L(p)\to0$
as $p\to1$.
In particular, the probability $p_n$ is well defined.
It is clear from the definitions of $L(p)$ and $p_n$ and from the RSW
theorem that, for positive integers $k$ and $l$,
there exists $\delta_{k,l}>0$ such that, for any positive integer $n$
and for all $p\in[p_c,p_n]$,
\[
\PP(\mbox{there is a $p$-open horizontal crossing of }[0,kn]\times
[0,ln]) >\delta_{k,l}
\]
and
\[
\PP\bigl(\mbox{there is a $p$-closed horizontal dual crossing of
}([0,kn]\times[0,ln])^*\bigr)>\delta_{k,l}.
\]
By the FKG inequality and a standard gluing argument \cite{Grimmett},
Section 11.7, we get that, for positive integers $n$ and
$k\geq2$
and for all $p\in[p_c,p_n]$,
\[
\PP\bigl(\operatorname{Ann}(n,kn)\mbox{ contains a $p$-open circuit around the origin}\bigr)
>(\delta_{k,k-2})^4
\]
and
\begin{eqnarray*}
&&\PP\bigl(\operatorname{Ann}(n,kn)^*\mbox{ contains a $p$-closed dual circuit around the
origin}\bigr)
\\
&&\qquad
>(\delta
_{2k,k-1})^4.
\end{eqnarray*}

%
\subsection{Properties of correlation length}

We give the following results without proofs.

\begin{longlist}[(1)]
%
\item[(1)]
Reference \cite{kesten}, Theorem 2.
There is a constant $D_1<\infty$ such that, for all $p>p_c$,
%
%
\begin{equation}\label{thetacorrineq}
\theta(p)
\leq
\PP\bigl[0 \stackrel{p}{\longleftrightarrow}\partial B(L(p))\bigr]
\leq
D_1 \PP\bigl[0\stackrel{p_c}{\longleftrightarrow}\partial
B(L(p))\bigr],
\end{equation}
where $\theta(p)=\PP(0\stackrel{p}{\longleftrightarrow}\infty)$ is
the percolation function for Bernoulli percolation.
\item[(2)]
Reference \cite{Nguyen}, Section 4.
There is a constant $D_2>0$ such that, for all $n\geq1$,
%
%
\begin{equation}\label{eqNguyen}
{\mathbb P}\bigl(B(n)\stackrel{p_n}{\longleftrightarrow}\infty\bigr)\geq D_2.
\end{equation}
%
%
\item[(3)]
For any $n \geq1$ and $p \in[0,1]$, let $B_{n,p}$ be the event that
there is a $p$-closed circuit around the origin in the dual lattice
with radius at least $n$.
There exist constants $D_3<\infty$ and $D_4>0$ such that for all $p> p_c$,
%
%
\begin{equation}\label{eqExpDecay}
{\mathbb P}(B_{n,p})
\leq
D_3 \exp\biggl\{-D_4\frac{n}{L(p)}\biggr\}.\vadjust{\goodbreak}
\end{equation}
Equation (\ref{eqExpDecay}) follows, for example, from \cite{Jarai},
(2.6) and (2.8) (see also \cite{Nolin}, Lemma 37 and Remark 38).
%
\item[(4)]
There exist constants $D_5>0$ and $D_6<\infty$ such that for all
$m,n\geq1$,
%
%
\begin{equation}\label{ineqP}
D_5\biggl|\log\frac{m}{n}\biggr| \leq\biggl|\log\frac
{p_m-p_c}{p_n-p_c}\biggr| \leq D_6\biggl|\log\frac{m}{n}\biggr|.
\end{equation}
This is a consequence of \cite{Nolin}, Proposition 34, and a priori
bounds on the 4-arm exponent.
\end{longlist}

\section{Properties of the outlet variables}\label{sec:outlets}
In this section we describe several important properties of the
variables $(O_k)$. We first recall the following theorem from \cite
{DS} that gives $k$-independent bounds on all of their moments.
\begin{theorem}\label{outletmomentthm}
There exists $c_1<\infty$ such that for all $t,k \geq1$,
%
%
\begin{equation}\label{eq: momentbound1}
\E(O_k^t) \leq(c_1t)^{3t}.
\end{equation}
\end{theorem}

One crucial feature of the invasion process that allows us to prove
limit theorems is its \textit{renewal structure}. To describe this, we
make a couple of definitions. For $k,m \geq1$ and $1 \leq l \leq
\infty$, let $\mathcal{G}(k,l,m)$ be the graph of the invasion
process that invades the entire box $B(2^{k-m})$ at step 1 [we take
$B(2^{k-m})$ to be the origin if $k < m$], then proceeds with the usual
invasion rules and stops when it invades any vertex of $\partial
B(2^{k+l+m})$. In the case that $l=\infty$, we allow the invasion to
run for all of time. Write $\mathcal{O}$ for the set of all outlets of
$\mathcal{S}$, and write $\mathcal{O}(k,l,m)$ for the set of all
outlets of $\mathcal{G}(k,l,m)$.
In the case of $\mathcal{O}(k,l,m)$, the outlets are defined in the
same way as in $\mathcal{O}$; however, note that if the graph
$\mathcal{G}(k,l,m)$ is finite (which corresponds to the case of
finite $l$), some of its outlets may have weight below $p_c$.

For the next theorem, when $l=\infty$, $\operatorname{Ann}(m,l)$ will mean $B(m)^c$.
%
\begin{theorem}[(Renewal structure of the invasion)]\label{thm:renewal}
There are constants $C<\infty$ and $\delta>0$ such that for all $k,m
\geq1$ and $1 \leq l \leq\infty$,
\[
\P\bigl(\mathcal{S} \cap \operatorname{Ann}(2^k,2^{k+l}) \neq\mathcal{G}(k,l,m)\cap
\operatorname{Ann}(2^k,2^{k+l})\bigr) < C\exp(-\delta m)
\]
and
\[
\P\bigl(\mathcal{O} \cap \operatorname{Ann}(2^k,2^{k+l}) \neq\mathcal{O}(k,l,m)\cap
\operatorname{Ann}(2^k,2^{k+l})\bigr) < C\exp(-\delta m).
\]
\end{theorem}
\begin{pf}
Clearly it suffices to prove the theorem for $m>4$. We first consider
the case that $k \geq m$ and $l<\infty$. Observe that $\mathcal{S}
\cap \operatorname{Ann}(2^k,2^{k+l}) = \mathcal{G}(k, l,m)\cap \operatorname{Ann}(2^k,2^{k+l})$ and
$\mathcal{O} \cap \operatorname{Ann}(2^k,2^{k+l}) = \mathcal{O}(k,l,m)\cap
\operatorname{Ann}(2^k,2^{k+l})$ if (1) there exists a $p_c$-open circuit\vspace*{2pt} around the
origin in $\operatorname{Ann}(2^{k-m},2^k)$,
(2)~there exists a $p_{2^{k+l+m/4}}$-closed dual circuit around the
origin in the annulus $\operatorname{Ann}(2^{k+l},2^{k+l+m/4})^*$,
(3) there exists a $p_c$-open circuit\vspace*{1pt} around the origin in
$\operatorname{Ann}(2^{k+l+m/2},2^{k+l+m})$
and (4) the open circuit from (3) is connected by a\vadjust{\goodbreak}
$p_{2^{k+l+m/4}}$-open path to infinity. (See Figure \ref{fig:renewal}
for an illustration of the intersection of these four events.)
Indeed, the first condition implies that in the exterior of the
$p_c$-open circuit from (1), $\mathcal G(k,l,m)$ is a subset of
$\mathcal S$.
The remaining conditions (2)--(4) imply the existence of an edge $e$ in
$\operatorname{Ann}(2^{k+l},2^{k+l+m})$, lying in the closure of the exterior of the
closed circuit from (2), such that $e \in\mathcal{O} \cap\mathcal
{O}(k,l,m)$, and both invasion processes invade $e$ before any vertex
of $\partial B(2^{k+l+m})$.
Therefore, once this outlet is invaded (by either of the two invasion
processes), the set of invaded edges in the interior of the closed
circuit from (2) does not change anymore. The RSW theorem and (\ref
{eqExpDecay}) imply that the probability that any of (1)--(4) does not
hold is bounded from above by $C \exp(-\delta m)$ uniformly in $k$.

%
%
\begin{figure}

\includegraphics{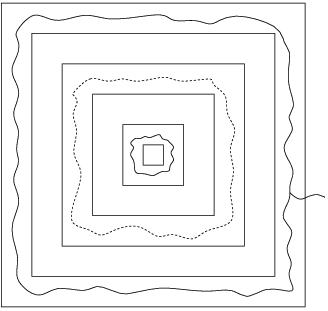}

\caption{The event in the proof of
Theorem \protect\ref{thm:renewal} (in the case $k \geq m$ and
$l<\infty$). The
boxes, in order from smallest to largest, are $B(2^{k-m})$, $B(2^k)$,
$B(2^{k+l})$, $B(2^{k+l+m/4})$, $B(2^{k+l+m/2})$ and $B(2^{k+l+m})$.
(Boxes are not drawn to scale.) The dotted path is
$p_{2^{k+l+m/4}}$-closed, the path to infinity is
$p_{2^{k+l+m/4}}$-open and the other two circuits are $p_c$-open. If
all these paths exist, the sets ${\mathcal S}$ and ${\mathcal
G}(k,l,m)$ coincide in $\operatorname{Ann}(2^k,2^{k+l})$.} \label{fig:renewal}
\end{figure}
%

In the case that $k < m$ and $l<\infty$, we exclude condition (1) from
the above argument. In the case $k < m$ and $l=\infty$ there is
nothing to prove. If $k \geq m$ and $l=\infty$ we argue using only
condition (1).
\end{pf}
\begin{remark}
Similar ideas were used in the proof of the upper bound in Theorem 1.4
in \cite{DS}. Note that there is
a typo there in the definition of $X^n_i$. It should be specified that
$X^n_i$ counts only disconnecting edges with weights larger than $p_c$.
\end{remark}

We now present corollaries of Theorem \ref{thm:renewal} that will help
in the proofs of the next section. The first two are about mixing
properties of the sequence $(X_k)$. Recall the notation that $a_k = \E
O_k$ and let $X_k = O_k - a_k$. For any \mbox{$m_1 \leq m_2$}, let $\Sigma
_{m_1}^{m_2}$ be the sigma algebra generated by the variables $X_{m_1},
\ldots, X_{m_2}$. Write~$\Sigma_{m_1}$ for $\lim_{m_2 \to\infty}
\Sigma_{m_1}^{m_2}$ and $\Sigma^{m_2}$ for $\Sigma_1^{m_2}$. For $m
\geq0$, define the \textit{strong mixing coefficient}
%
%
\begin{equation}\label{eq: mixingdef}
\alpha(m) = {\sup_{k \geq1}} \sup_{A,B} |\P(A \cap B) - \P(A)\P
(B)|,
\end{equation}
where the supremum is over all $A\in\Sigma^k$ and $B \in\Sigma_{k+m}$.
\begin{corollary}\label{cor:mixing}
There exist constants $C<\infty$ and $\delta>0$ such that for all~$m$,
%
%
\begin{equation}\label{eq: mixingeq1}
\alpha(m) \leq C \exp(- \delta m).
\end{equation}
\end{corollary}
\begin{pf}
Clearly it suffices to prove the corollary for $m>4$. Fix $k\geq1$ and
let $A \in\Sigma^k$, $B\in\Sigma_{k+m}$. For $j=1, \ldots, k$, let
$\tilde Y_j$ be the number of outlets in $\mathcal{O}(0,k,\lfloor
m/2 \rfloor-1) \cap \operatorname{Ann}(2^{j-1},2^j)$ with weight $>p_c$, and for
$j\geq k+m$, let~$\tilde Y_j$ be the number of outlets in $\mathcal
{O}(k+m-1,\infty, \lfloor m/2 \rfloor-1) \cap \operatorname{Ann}(2^{j-1},2^j)$ with
weight $>p_c$. Let $Y_j = \tilde Y_j - a_j$. By Theorem \ref
{thm:renewal}, there exist constants $C_1<\infty$ and $\delta_1>0$
such that for all $k \geq1$, $m > 4$,
\[
\P(A_{k,m}) \geq1 - C_1 \exp(-\delta_1 m),
\]
where $A_{k,m}$ is the event that $X_j = Y_j$ for all $j \leq k$ and
for all $j \geq k+m$.

Because $A \in\Sigma^k$, there exists a Borel set $A' \subset\R^k$
such that $A$ is the event that $(X_1, \ldots, X_k) \in A'$.
Similarly, because $B \in\Sigma_{k+m}$, there exists a Borel set $B'
\subset\R^\infty$ (with the product topology) such that $B$ is the
event that $(X_{k+m}, \ldots) \in B'$. Define $A_Y$ as the event that
$(Y_1, \ldots, Y_k) \in A'$ and $B_Y$ as the event that
$(Y_{k+m},\ldots) \in B'$. Because $A_Y$ and $B_Y$ are independent,
%
%
\begin{equation}\label{eq:indep}
| \P(A_Y\cap B_Y) - \P(A_Y)\P(B_Y) | = 0.
\end{equation}
Also, when $A_{k,m}$ occurs, the events $A$ and $A_Y$ (resp., $B$ and
$B_Y$) are identical, so
\[
| \P(A \cap B) - \P(A_Y \cap B_Y) | \leq\P(A_{k,m}^c) \leq C_1 \exp
(-\delta_1 m)
\]
and
\begin{eqnarray*}
| \P(A)\P(B) - \P(A_Y)\P(B_Y)| &\leq& \P(A)|\P(B)-\P(B_Y)| \\
&&{} +  \P(B_Y)|\P(A)-\P(A_Y)| \\
&\leq& |\P(B)-\P(B_Y)| + |\P(A)-\P(A_Y)| \\
&\leq& 2C_1 \exp(-\delta_1 m).
\end{eqnarray*}
Combining the two above inequalities with (\ref{eq:indep}) gives the corollary.
\end{pf}

Now that we have a bound on the decay of the sequence $(\alpha(m))$,
we can relate this to the decay of covariances using the following
classical result.
\begin{corollary}[(\cite{Davydov}, (2.2))]\label{cor:cov}
Let $k,m \geq1$ and let $f$ and $g$ be functions such that $f$ is
$\Sigma^k$-measurable and $g$ is $\Sigma_{k+m}$-measurable. Suppose
that $1/p + 1/q < 1$ and that the moments $\E|f|^p$ and $\E|g|^q$
exist. Then
%
%
\begin{equation}\label{eq: covariance}
| \E fg - \E f \E g| \leq12 [\E|f|^p]^{1/p} [\E
|g|^q]^{1/q} [\alpha(m)]^{1-1/p-1/q}.
\end{equation}
\end{corollary}
\begin{pf}
For completeness, we will outline the proof in the \hyperref[app]{Appendix}.
\end{pf}


Corollaries \ref{cor:mixing} and \ref{cor:cov} tell us that the
variables $(X_k)$ are very weakly dependent. This is one main
ingredient for proving the CLT and SLLN for this sequence. In the first
part of the following corollary, we will bound moments of the sums
$(\sum_{k=1}^n X_k)_n$. This is the second main ingredient necessary
for proving the CLT. The second part of the corollary will control
fluctuations of the sums and will be useful in proving the SLLN.
\begin{corollary}\label{cor:momentmixing}
The following statements hold.
\begin{longlist}[(2)]
\item[(1)] For each $0 \leq t \leq4$, there exists $D(t)<\infty$ such that
for all $k \geq1$ and $m \geq0$,
\[
\E\Biggl| \sum_{j=k}^{k+m} X_j \Biggr|^t \leq D(t) m^{t/2}.
\]
\item[(2)] There exists $C<\infty$ such that for any $\lambda>0$ and $n
\geq1$,
\[
\P\Biggl( \max_{1 \leq i \leq n} \Biggl| \sum_{k=1}^i X_k \Biggr|
\geq\lambda\Biggr) \leq\frac{Cn}{\lambda^2} + \frac{C\sqrt
n}{\lambda}.
\]
\end{longlist}
\end{corollary}
\begin{pf}
We will begin with the proof of the first statement. It suffices to
consider $t=4$ because for $t<4$ we can use Jensen's inequality to
reduce to this case. The statement will follow from Proposition 2.2 of
\cite{Withers}, which we state below as Lemma \ref{lem: withers22}.
For the statement, we need some definitions. For $0\leq k < n$, define
\[
c^{(13)}(k,n) = \max_{1 \leq x_1, x_2 = x_1+k \leq x_3\leq x_4\leq n}
\E X_{x_1}X_{x_2}X_{x_3}X_{x_4}
\]
and
\[
c^{(31)}(k,n) = \max_{1 \leq x_1\leq x_2 \leq x_3, x_4=x_3+k\leq n} \E
X_{x_1}X_{x_2}X_{x_3}X_{x_4}.
\]
Also set
\[
c(k;1,3) = \sup_{n \geq k} \bigl[c^{(13)}(k,n)+c^{(31)}(k,n)\bigr].
\]
\begin{lemma}\label{lem: withers22}
Suppose that $\sup_{k\geq1} \E X_k^4 < \infty$ and
%
%
\begin{equation}\label{eq: witherscondition1}
\sum_{k=0}^m (k+1)c(k;1,3) = O(m^\gamma) \qquad\mbox{as } m \to\infty
\mbox{ for } \gamma\geq0.
\end{equation}
Then
\[
\sup_{b\geq0} \E(X_b + \cdots + X_{b+a})^4 = O(a^{2+\gamma}) \qquad\mbox{as } a \to\infty.
\]
\end{lemma}

We make the choice $\gamma=0$. The condition $\sup_{k \geq1} \E
X_k^4< \infty$ holds from Theorem \ref{outletmomentthm}. As for
(\ref{eq: witherscondition1}), it is not difficult to see that it will
hold as long as we show that there exist constants $C_1<\infty$ and
$\delta_1>0$ such that for any $m \geq1$ and for any natural numbers
$i_1, \ldots, i_4$ such that the distance from $i_1$ to the set $\{
i_2, i_3, i_4\}$ is at least equal to $m$,
%
%
\begin{equation}\label{eq: ourcondition1}
|\E X_{i_1} \cdots X_{i_4}| \leq C_1 \exp(-\delta_1 m).
\end{equation}
Condition (\ref{eq: ourcondition1}) holds by Corollary \ref{cor:cov}.
To show this, suppose that $i_1 \leq i_2 \leq i_3 \leq i_4$ (the other
cases are handled similarly). We make the choices $f=X_{i_1}$ and
$g=X_{i_2}X_{i_3}X_{i_4}$, with $p=2$ and $q=4$. From Theorem \ref
{outletmomentthm}, there exists $C_2$ such that for all $(i_j)$, both
$(\E g^4)^{1/4} \leq C_2$ and $(\E f^2)^{1/2} \leq C_2$. Since $\E f =
0$, Corollary \ref{cor:cov} gives
\[
|\E X_{i_1} \cdots X_{i_4} | \leq C_2^2 \alpha(m)^{1/4}.
\]
Bounding $\alpha(m)$ using Corollary \ref{cor:mixing} shows (\ref
{eq: ourcondition1}) and completes the proof of the first statement of
Corollary \ref{cor:momentmixing}.

We now prove the second statement. It is the same as the proof of~Lem\-ma~2.2
in \cite{Davydov}. Let $A_n$ be the event in the statement,
and write $S_n = \sum_{k=1}^n X_k$. Let
\[
A_n^i = \{ |S_j|< \lambda\mbox{ for } j = 1, \ldots, i-1 \mbox
{ but } |S_i| \geq\lambda\}.
\]
Similarly to the proof of Kolmogorov's maximal inequality for
independent random variables, one can show that
%
%
\begin{equation}\label{eq: kolm}
\P(A_n) \leq\frac{1}{\lambda^2} \Biggl( \E S_n^2 + 2 \sum_{i=1}^n
\E\bigl[ I[A_n^i] S_i(S_n-S_i) \bigr] \Biggr).
\end{equation}
By the first part of this corollary, $\E S_n^2 \leq C_3n$. Next, write
the summand as
\[
\E\bigl[ I[A_n^i] X_i(S_n-S_i) \bigr] + \sum_{j=i+1}^n\E[
I[A_n^i] S_{i-1}X_j].
\]
The absolute value of the first term is bounded by
\[
\Biggl| \sum_{k=i+1}^n \E[ I[A_n^i] X_iX_k] \Biggr|
\leq\sum_{k=i+1}^n C_4 [\alpha(k-i)]^{1/4} \leq C_5,
\]
where we use Corollary \ref{cor:cov} with $f=I[A_n^i]X_i$ and $g=X_k$,
with $p=2$ and $q=4$ (bounding the moments using Theorem \ref
{outletmomentthm}) in the first inequality. For the second term we also
use Corollary \ref{cor:cov} but choose $f=I[A_n^i] S_{i-1}$ and
\mbox{$g=X_j$}, with $p=2$ and $q=4$. This produces the bound
\begin{eqnarray*}
C_6 \sum_{j=i+1}^n (\E[S_{i-1}I[A_n^i]]^2)^{1/2} [\alpha(j-i)]^{1/4}
&\leq& C_6 \lambda\sqrt{\P(A_n^i)} \sum_{j=1}^\infty[\alpha
(j)]^{1/4} \\
&\leq& C_7 \lambda\sqrt{\P(A_n^i)}.
\end{eqnarray*}
Summing over $i$ and using Jensen's inequality with the square root
function (recalling that the events $A_n^i$ are disjoint in $i$), we
see that the sum in (\ref{eq: kolm}) is no bigger than
\[
2C_5 n + 2C_7\lambda n \Biggl( \frac{1}{n} \sum_{i=1}^n \sqrt{\P
(A_n^i)} \Biggr) \leq C_8n + C_9\lambda\sqrt n.
\]
Putting both this bound and the one on $\E S_n^2$ into (\ref{eq:
kolm}) finishes the proof.
\end{pf}

The following corollary shows a way to construct a sequence of $c\log
n$-dependent random variables $(\tilde O_k)$ related to $(O_k)$. We
will not use this sequence in the rest of the paper; however, the
proofs of the CLT and the SLLN given in Section \ref{sec:CLT} can be
replaced by ones that make reference to neither \cite{Davydov} nor
\cite{Withers} but that come from corresponding statements involving
independent random variables by using the $\tilde O_k$'s. An example of
such an approach is the proof of Theorem 1.4 in~\cite{DS}.
\begin{corollary}\label{cor:mndependent}
For any $\gamma>0$, there exists $c<\infty$ such that for all $n \geq
1$, defining $m_n = c\log n$, with probability at least $1 -
cn^{-\gamma}$, all
random variables $O_{m_n+1},\ldots,O_n$ are equal to some random
variables $\tilde
O_{m_n+1},\ldots,\tilde O_n$, which are $m_n$-dependent and satisfy
Theorem \ref{outletmomentthm}.
\end{corollary}
\begin{pf}
Let $c$ be an integer to be chosen later and let $k \geq c \log n$. We
define $\tilde O_k$ as the number of outlets in $\mathcal
{O}(k-1,1,\lfloor m_n/2 \rfloor-1) \cap \operatorname{Ann}(2^{k-1},2^k)$ with
$\mathrm{weight}>p_c$. The reader may verify that exactly the same argument used in
\cite{DS} for the proof of Theorem \ref{outletmomentthm} applies to
each~$\tilde O_k$. Also, the variables~$(\tilde O_k)$ are
obviously $m_n$-dependent. By Theorem \ref{thm:renewal}, there exist
$C<\infty$ and $\beta>0$ such that for any $k \geq c\log n$,
\[
{\mathbb P}(\tilde O_k\neq O_k) \leq C n^{-\beta},
\]
where $\beta\to\infty$ as $c \to\infty$. Therefore,
\[
{\mathbb P}(\tilde O_k\neq O_k \mbox{ for some } k\in[c\log n,n])
\leq Cn^{1-\beta}.
\]
\upqed\end{pf}

\section{CLT and SLLN for the outlets}\label{sec:CLT}

\subsection{\texorpdfstring{Proof of Theorem \protect\ref{thm:VarOn}}{Proof of Theorem 1}}
First we will show the statement about the~$a_k$'s. Theorem \ref
{outletmomentthm} implies the upper bound on $a_k$, so we need only
show the lower bound. The proof is similar to the first part of
Theorem 1.4 in \cite{DS}. For $k \geq1$, let $A_k$ be the event that
(a) there is a $p_{2^k}$-closed circuit around the origin in
$\operatorname{Ann}(2^{k-1},2^{k})$, (b) there is a $p_{2^k}$-open circuit in
$\operatorname{Ann}(2^{k-1},2^{k})$ and (c)~the circuit from (b) is connected by a
$p_{2^k}$-open path to infinity. By the RSW theorem and (\ref
{eqNguyen}), there exists $C_1>0$ such that for all $k$,
\[
\P(A_k) > C_1.
\]
But $A_k$ implies the event $\{ O_k \geq1\}$, so
\[
a_k = \E O_k \geq\P(A_k) \geq C_1.
\]

We move on to the statement about $b(n)$. The upper bound follows from
the case $t=2$ of the first statement in Corollary \ref
{cor:momentmixing}, so we will focus on the lower bound. Let $k$ be an
integer between 1 and $n$ and let $L_n := \log n$. For $i=1,\ldots, 5$
define $q_k(i) = p_c+ i(p_{2^k}-p_c)$. We define $A_{n,k}$ as the event that:
\begin{longlist}[(2)]
\item[(1)] there is an edge $e_1$ in $\operatorname{Ann}(2^{k+1},2^{k+2})$, with weight
between $q_k(4)$ and~$q_k(5)$, which is connected by a $p_c$-open path
to a $p_c$-open circuit around the origin that is in $\operatorname{Ann}(2^k,2^{k+1})$;
\item[(2)] the endpoints of $e_1^*$ are connected by a $q_k(5)$-closed dual
path in $\operatorname{Ann}(2^{k+1}$, $2^{k+2})^*$ such that the union of this path and
$e_1^*$ encloses the origin;
\item[(3)] there is an edge $e_2$ in $\operatorname{Ann}(2^{k+2},2^{k+3})$, with weight in
$[q_k(1),q_k(2)] \cup[q_k(3),q_k(4)]$, which is connected by a
$p_c$-open path to an endpoint of $e_1$;
\item[(4)] the endpoints of $e_2^*$ are connected by a $q_k(5)$-closed dual
path in $\operatorname{Ann}(2^{k+2}$, $2^{k+3})^*$ such that the union of this path and
$e_2^*$ encloses the origin;
\item[(5)] there is an edge $e_3$ in $\operatorname{Ann}(2^{k+3},2^{k+4})$ with weight in
$[q_k(2),q_k(3)]$, which is connected by a $p_c$-open path to an
endpoint of $e_2$;
\item[(6)] the endpoints of $e_3^*$ are connected by a $q_k(5)$-closed dual
path in $\operatorname{Ann}(2^{k+3}$, $2^{k+4})^*$ such that the union of this path and
$e_3^*$ encloses the origin;
\item[(7)] an endpoint of $e_3$ is connected by a $q_k(1)$-open path to
$\partial B(2^{k+L_n})$.
\end{longlist}
Notice that if $A_{n,k}$ occurs with edges $e_1 -e_3$, it cannot occur
with any other edges. It follows from \cite{DSV}, Lemma 6.3, and RSW
arguments (similar to the proof of \cite{DSV}, Corollary 6.2) that
there exists $C_2>0$ such that for any $n,k,$
%
%
\begin{equation}\label{eq:ank}
\P(A_{n,k}) \geq C_2.
\end{equation}
Since, in addition, the events $A_{n,k}$ are $L_n$-dependent for fixed
$n$, there exists $C_3>0$ such that
%
%
\begin{eqnarray}\label{eq: largedev}
\P(A_{n,3k} \mbox{ occurs for at least } C_3 n \mbox{ values of } k
\in[1,n/3]) \to1 \nonumber\\[-8pt]\\[-8pt]
&&\eqntext{\mbox{as } n \to\infty.}
\end{eqnarray}
To see this, we will give the proof in Theorem 1.4 of \cite{DS}. Let
$j$ be an integer between 1 and $L_n$, and define $B_i^j =
A_{n,3(j+iL_n)}$. Note that the events $(B_i^j)_{i=0}^{\lfloor n/3L_n
\rfloor- 1}$ are independent. Therefore we may use Lemma 5.2 from
\cite{DS}. Its proof is standard, so we omit it.
\begin{lemma}\label{lem:DS}
Let $c>0$. There exist $\alpha>0$ and $\beta<1$ depending on $c$ with
the following property. If $Y_i$ are independent $0/1$ random variables
(not necessarily identically distributed) with $\P(Y_i=1)>c$ for all
$i$, then for all $n$,
\[
\P\Biggl( \sum_{i=1}^n Y_i < \alpha n \Biggr) < \beta^n.
\]
\end{lemma}

In view of this lemma and (\ref{eq:ank}), there exist $\alpha>0$ and
$\beta<1$ such that for any $n$ and $1 \leq j \leq L_n$,
\[
\P\Biggl( \sum_{i=0}^{\lfloor n/3L_n \rfloor- 1} I[B_i^j] < \frac
{\alpha n}{3L_n} \Biggr) < \beta^{n/3L_n}.
\]
Therefore,
\begin{eqnarray*}
&& \P\Biggl(\sum_{j=1}^{L_n} \sum_{i=0}^{\lfloor n/3L_n \rfloor-1}
I[B_i^j] < \alpha n/3 \Biggr) \\
&&\qquad\leq \P\Biggl(\sum_{i=0}^{\lfloor n/3L_n \rfloor- 1} I[B_i^j] <
\alpha n/3L_n \mbox{ for some } j \in[1,L_n]\Biggr) \\
&&\qquad\leq L_n \beta
^{n/3L_n},
\end{eqnarray*}
which converges to 0 as $n \to\infty$. This proves (\ref{eq: largedev}).

Define $\tilde A_{n,k}$ the same way as we defined $A_{n,k}$ except
that in item 7, the $q_k(1)$-open path\vspace*{1pt} connects $e_3$ to infinity. (See
Figure \ref{fig:VarOn} for an illustration of the event $\tilde
A_{n,k}$.) Note that if $\tilde A_{n,k}$ occurs, then $e_1$ and $e_3$
%
%
\begin{figure}

\includegraphics{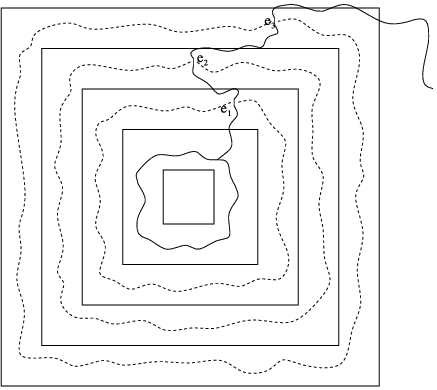}

\caption{The event $\tilde A_{n,k}$. The boxes, in order from
smallest to largest, are $B(2^{k+i})$ for $i=0,\ldots, 4$. The dotted
paths are $q_k(5)$-closed, the path from $e_3$ to infinity is
$q_k(1)$-open and all the other paths are $p_c$-open. The weight of
$e_1$ is in $[q_k(4),q_k(5)]$, the weight of $e_3$ is in
$[q_k(2),q_k(3)]$ and the weight of $e_2$ is in $[q_k(1),q_k(2)]\cup
[q_k(3),q_k(4)]$. The edges $e_1$ and~$e_3$ are outlets. The edge $e_2$
is an outlet if and only if its weight is in $[q_k(3),q_k(4)]$.}
\label{fig:VarOn}
\vspace*{-3pt}
\end{figure}
are outlets, and~$e_2$ is an outlet if and only if its weight is in
$[q_k(3),q_k(4)]$. If $A_{n,k}$ occurs but $\tilde A_{n,k}$ does not,
then there exists a $q_k(1)$-closed dual circuit around the origin with
radius at least $2^{k+L_n}$. By (\ref{eqExpDecay}), there exist
constants $C_4<\infty$ and $C_5>0$ such that for all $n,k$, $\P
(A_{n,k} \setminus\tilde A_{n,k}) \leq C_4\exp(-C_5 2^{k+L_n}/2^k)$, so
\[
\P(A_{n,k} \setminus\tilde A_{n,k} \mbox{ occurs for some } k \in
[1, n]) \to0.
\]
Therefore we may find $C_6>0$ such that for all $n$,
\[
\P(\tilde A_{n,3k} \mbox{ occurs for at least } C_3 n \mbox{ values
of } k \in[1,n/3]) > C_6.
\]

Call $A$ the above event whose probability is bounded below by $C_6$.
On the event~$A$, we define the vector $\vec{f} = (f_1, \ldots,
f_{\lfloor C_3n \rfloor})$ whose entries are the first $\lfloor C_3n
\rfloor$ edges $e$ (ordered from distance to the origin) such that
there exist edges $\bar e_1$ and\vadjust{\goodbreak} $\bar e_3$ such that $\bar e_1$, $e$
and $\bar e_3$ satisfy the properties of $e_1, e_2$ and~$e_3$,
respectively, in the definition of $A_{n,3k}$ for some $k \in[1,
n/3]$.\vspace*{-2pt} Write~$O_{\vec f}(n)$ for the number of outlets that appear in
the vector $\vec f$, and write $U_{\vec f}$\vspace*{-2pt} for the number of outlets
in $B(2^n)$ that do not appear in $\vec f$. At least one of $\{U_{\vec
f} +(C_2n)/2 \geq a(n) \}$ or $\{U_{\vec f} + (C_2n)/2 \leq a(n) \}$
has probability at least $C_6/3$. Let us assume that it is the first
event; if it is the other then the subsequent argument can be easily
modified. Write $B = \{U_{\vec f} + (C_2n)/2 \geq a(n)\}$. Since
$U_{\vec f}$ is defined only on $A$, we have $B \subset A$.

Associated to each $f_k$ in $\vec f$ in the definition of $A_{n,k}$ are
two intervals $I_k(1)=[q_k(1),q_k(2)]$ and $I_k(2)=[q_k(3),q_k(4)]$.
Let\vspace*{2pt} $\eta(\vec f)$ be the configuration of weights outside of $\vec
f$. If $\vec f$ and $\eta(\vec f)$ are fixed, then the variable
$U_{\vec f}$\vspace*{2pt} is a~constant function of the weights $\tau_{f_k}$. Also,
when these variables are fixed, $O_{\vec f}(n)$ is equal to the number
of values of $k \in[1,\lfloor C_3 n \rfloor]$ such that $\tau_{f_k}
\in I_k(2)$. Since the lengths of $I_k(1)$ and $I_k(2)$ are equal, the
distribution of $O_{\vec f}(n)$ conditioned on $\vec f$, $\eta(\vec
f)$ and $B$ is Binomial$(\lfloor C_3n \rfloor, 1/2)$. If $Y$ is an
independent variable with this distribution, then
\begin{eqnarray*}
\P\bigl( |{\mathbf O}(n) - a_n| \geq\sqrt n \bigr) &\geq& \E\bigl[ \P\bigl(
O_{\vec f} (n) \geq(C_3n)/2 + \sqrt n \mid B,\vec f, \eta(\vec
f)\bigr)\bigr] \P(B) \\
&\geq& (C_6/3) \P\bigl(Y \geq(C_3n)/2 + \sqrt n\bigr),
\end{eqnarray*}
which is bounded below uniformly in $n$. This completes the
proof.\vadjust{\goodbreak}


\subsection{\texorpdfstring{Proof of Theorem \protect\ref{thm:cltokthm}}{Proof of Theorem 2}}

\mbox{}

\begin{pf*}{Proof of the CLT}
We will apply Theorem 2.1 of \cite{Withers}. To state that theorem, we
need to introduce the notion of $l$-mixing. For $k\geq0$, $n \geq1$
and $u \in\R$, set
%
%
\begin{equation}\label{eq: ln}
l_n(k,u) = \max_{1\leq j \leq n-k} \sup|\E[e^{iuP} e^{-iuF}] - \E
e^{iuP} \E e^{-iuF}|,
\end{equation}
where
\[
P = b(n)^{-1} \sum_{l=1}^j \delta_l X_l,\qquad F= b(n)^{-1} \sum
_{l=j+k}^n \delta_l X_l,
\]
and the supremum in (\ref{eq: ln}) is over all $\{\delta_l = 0 \mbox
{ or } 1\}$. Now for $k \geq0$ and $u \in\R$, set
\[
l(k,u) = \sup_{n \geq1} l_n(k,u).
\]
The sequence $(X_k)$ is called $l$-\textit{mixing} if for all real $u$,
$l(k,u) \to0$ as $k \to\infty$.
\begin{remark}
As mentioned in the discussion below Definition 2.2 in \cite{Withers},
the inequality
%
%
\begin{equation}\label{eq: domination}
l(k,u)\leq16\alpha(k)
\end{equation}
from page 307 in \cite{IL} holds for all $k \geq0$ and $u\in\R$, so
since $\alpha(k) \to0$ as $k \to\infty$, the sequence $(X_k)$ of
outlet variables is $l$-mixing.
\end{remark}

For $k \geq0$, define
\[
\tilde c(k) = \sup_{j \geq1} | \E X_jX_{j+k} |.
\]
The following is Theorem 2.1 of \cite{Withers}.
\begin{lemma}\label{lem: withers}
The following conditions are sufficient for
\[
\frac{\sum_{k=1}^nX_k}{b(n)} \Rightarrow N(0,1).
\]
For some $\varepsilon>0$ and $\gamma\geq0$,
%
%
\begin{equation}\label{eq: condition1}
\sup_{a \geq1} \E\Biggl| \sum_{k=a}^{a+b} X_k \Biggr|^{2 +
\varepsilon} = O(b^{1+\varepsilon/2+\gamma}) \qquad\mbox{as } b \to
\infty;
\end{equation}
the sequence $(X_k)$ is $l$-mixing and for all real $u$,
%
%
\begin{equation}\label{eq: condition2}
l(k,u) = o(k^{-\theta}) \qquad\mbox{as } k \to\infty, \mbox{ where }
\theta= 2\gamma/\varepsilon,
\end{equation}
and
%
%
\begin{equation}\label{eq: condition34}
b(n) \to\infty\qquad\mbox{as } n \to\infty\quad\mbox{and}\quad \sum
_{j=0}^\infty\tilde c(j) < \infty.
\end{equation}
\end{lemma}

To prove the CLT, we simply need to verify the conditions of Lemma \ref
{lem: withers}. Condition (\ref{eq: condition1}) holds with
$\varepsilon= 2$ and $\gamma=0$ by the first part of Corollary~\ref
{cor:momentmixing}, using $t=4$. Using (\ref{eq: domination}) and
Corollary \ref{cor:mixing}, we see that condition (\ref{eq:
condition2}) holds. Also, the first part of Corollary \ref
{cor:momentmixing} with $t=2$ shows the first part of condition~(\ref
{eq: condition34}). Finally, to verify the second part of~(\ref{eq:
condition34}), we appeal to Corollary~\ref{cor:cov} using $f=X_m$ and
$g=X_{m+k}$ (for fixed $m \geq1$ and $k \geq0$), with $p=2$ and
$q=4$. It follows that
\[
\tilde c(k) \leq C_1 \alpha(k)^{1/4}
\]
for some $C_1 <\infty$. In view of Corollary \ref{cor:mixing}, this
proves the second part of (\ref{eq: condition34}) and completes the
proof of the CLT.
\end{pf*}
\begin{pf*}{Proof of the SLLN}
For $i \geq1$, take $n_i = 2^i$. The second statement of
Corollary \ref{cor:momentmixing} implies that for any $\varepsilon> 0$,
\begin{eqnarray*}
\P\Bigl( {\max_{n_i \leq j \leq n_{i+1}}} | {\mathbf O}(j) - a(j) |
\geq\varepsilon n_i^r \Bigr) &\leq& \frac{C}{\varepsilon^2}
\biggl( \frac{n_{i+1}}{ n_i^{2r}} + \frac{\sqrt{n_{i+1}}}{n_i} \biggr) \\
&\leq& C_1\biggl( \frac{1}{(2^{2r-1})^i} + \frac{1}{(2^{r-1/2})^i}
\biggr).
\end{eqnarray*}
Since $r > 1/2$, this probability is summable in $i$. Since the
function $n^r$ is monotone, it follows that
\[
\sum_{i=1}^\infty\P\biggl( \max_{n_i \leq j \leq n_{i+1}} \frac
{|{\mathbf O}(j)-a(j)|}{j^r} \geq\varepsilon\biggr) < \infty.
\]
The Borel--Cantelli lemma finishes the proof.
\end{pf*}

\section{Further results for invasion percolation}\label{sec:further}

We begin with a lemma.
\begin{lemma}\label{lem:LD}
There exist constants $C<\infty$ and $\alpha>0$ such that for all $m,
n \geq1$,
\[
\P\bigl({\mathbf O}(n,n+m) \leq\alpha m\bigr) \leq C\exp(-m^\alpha),
\]
where ${\mathbf O}(n,n+m)$ is the number of outlets in $\operatorname{Ann}(2^n,2^{n+m})$.
\end{lemma}
\begin{pf}
The proof of the lower bound in Theorem 1.4 of \cite{DS} shows the
case $n=1$. For general $n$ the proof is similar. For $i,m \geq1$, let
$G_{i,m}$ be the event that there is no $p_{2^i}$-closed dual circuit
around the origin with radius larger than $2^{i+\log m}$, and let
$K_{i,m}$ be the event that (a) there exists a~$p_{2^i}$-closed dual
circuit ${\mathcal C}$ around the
origin in $\operatorname{Ann}(2^i,2^{i+1})^*$, (b) there exists a~$p_c$-open circuit
${\mathcal C}'$ around the origin
in $\operatorname{Ann}(2^i,2^{i+1})$ and (c) the circuit~${\mathcal C}'$ is connected
to $\partial B(2^{i+\log m})$ by a $p_{2^i}$-open path. By the RSW
theorem and~(\ref{eqNguyen}), there exists $C_1>0$ such that for all
$i \geq0$ and $m \geq1$,
\[
\P(K_{i,m}) \geq C_1.
\]

Now let\vspace*{-1pt} $j$ be an integer between $1$ and $\log m$, and define the
event $K_{i,m}^j = K_{n+j+i\lfloor\log m\rfloor, m}$.
Note that\vspace*{1pt} for fixed $j$, the events $(K_{i,m}^j)_{i=0}^{\lfloor m/\log
m\rfloor-1}$ are independent. Therefore we can apply Lemma \ref
{lem:DS} to deduce that there exist $\alpha_0>0$ and $\beta_0<1$ such
that for any $m,n$ and $1 \leq j \leq\log m$,
\[
\P\Biggl( \sum_{i=0}^{\lfloor m / \log m \rfloor- 1} I[K_{i,m}^j] <
\frac{\alpha_0 m}{\lfloor\log m\rfloor} \Biggr) < \beta_0^{m/\log
m}.
\]
Therefore,
\begin{eqnarray*}
&& \P\Biggl( \sum_{j=1}^{\lfloor\log m\rfloor} \sum
_{i=0}^{\lfloor m / \log m \rfloor-1} I[K_{i,m}^j] < \alpha_0 m
\Biggr) \\
&&\qquad\leq \P\Biggl( \sum_{i=0}^{\lfloor m/ \log m\rfloor-1}
I[K_{i,m}^j] < \alpha_0 m /\lfloor\log m\rfloor\mbox{ for some } j
\in[1,\log m] \Biggr) \\
&&\qquad\leq \log m \beta_0^{m/\log m} \leq C_2 \exp(-m^{\alpha})
\end{eqnarray*}
for some $C_2<\infty$ and $\alpha>0$. By (\ref{eqExpDecay}), we also
have the estimate
\[
\sum_{i=0}^m \P(G_{n+i,m}^c) \leq C_3\sum_{i=0}^m \exp(-C_4m) \leq
C_3\exp(-C_5m).
\]
Since the event $K_{i,m} \cap G_{i,m}$ implies $O_{i+1} \geq1$, we can
combine the above estimates to deduce
\[
\P\bigl({\mathbf O}(n,n+m) \leq\alpha_0m\bigr) \leq C_2\exp(-m^{\alpha}) +
C_3\exp(-C_5m),
\]
which implies the lemma.
\end{pf}

Recall the definitions of $Q_n$ and $T_n$ from Section \ref
{subsec:mainresults}. Since $a(n)\asymp n$, $T_n$ is comparable with $n$.
%
%
\begin{pf*}{Proof of Theorem \ref{thm:inverseCLT}}
It follows from the definition of $Q_n$, the CLT for ${\mathbf O}(n)$
and the fact that for any $x$, $\sigma_{n+x\sqrt n}/\sigma_n \to1$
as $n\to\infty$ that
\[
{\mathbb P}(Q_n < T_{n+ x\sigma_n}) \to\Phi(x),
\]
where $\Phi$ is the standard normal cumulative distribution function.
Recall that the $a_i$'s\vadjust{\goodbreak} [$a_i = a(i)-a(i-1)$] are uniformly bounded away
from 0 and~$\infty$ by Theorem~\ref{thm:VarOn}. Therefore,
\[
{\mathbb P}(Q_n < T_{n+ x\sigma_n}) =
{\mathbb P}\bigl(a(Q_n) < a(T_{n+ x\sigma_n})\bigr)=
{\mathbb P}\bigl(a(Q_n) < n+ x\sigma_n +r_n\bigr),
\]
where $r_n$ is uniformly bounded in $n$. It remains to prove the second
part of the proposition.
The first statement implies that for any $M>0$,
\[
{\mathbb E}\min\biggl\{\biggl(\frac{a(Q_n) - n}{\sigma_n}
\biggr)^2,M\biggr\} \to{\mathbb E}\min\{Z^2,M\},
\]
where $Z$ is a standard normal random variable.
Therefore, it suffices to show that for any $\varepsilon>0$ there
exists $C_1>0$ such that
\[
\limsup_{n\to\infty}{\mathbb E}\biggl(\frac{a(Q_n) - n}{\sigma
_n}\biggr)^2I\bigl(|a(Q_n) - n| > C_1\sigma_n\bigr) < \varepsilon.
\]
%
This will follow if we show that there exists $C_2$ such that for all $n$,
\[
{\mathbb E}\biggl(\frac{a(Q_n) - n}{\sigma_n}\biggr)^4 < C_2.
\]
In other words, we need to show that ${\mathbb E}(a(Q_n) - n)^4 = O(n^2)$.
Since the~$a_i$'s are uniformly bounded away from 0 and $\infty$, it
suffices to show that ${\mathbb E}(Q_n - T_n)^4 = O(n^2)$.
For $c>0$, consider
\[
A_n = \bigl\{{\mathbf O}(n,n+k)>ck, {\mathbf O}(n-k,n)>ck \mbox{ for
all } k\geq\sqrt n\bigr\}.
\]
It follows from Lemma \ref{lem:LD} that there exists $c>0$ such that
\[
{\mathbb P}(A_n^c) \leq C_3 \exp(-n^c).
\]
We write
\[
{\mathbb E}(Q_n-T_n)^4 = {\mathbb E}(Q_n-T_n)^4I(Q_n \leq C_4n) +
{\mathbb E}(Q_n-T_n)^4I(Q_n > C_4n).
\]
If $C_4$ is large enough, ${\mathbb E}(Q_n-T_n)^4I(Q_n > C_4n) =
o(n^2)$ (One can write $I(Q_n > C_4n)$ as $\sum_{k=1}^\infty I(Q_n \in
(kC_4n,(k+1)C_4n])$ and use Lemma \ref{lem:LD}).
We now bound the first expectation.
\begin{eqnarray*}
&& {\mathbb E}(Q_n-T_n)^4I(Q_n \leq C_4n) \\
&&\qquad\leq (C_4n)^4{\mathbb P}(A_{T_n}^c) + T_n^2 + {\mathbb
E}(Q_n-T_n)^4I\bigl(|Q_n-T_n|>\sqrt{T_n},A_{T_n}\bigr).
\end{eqnarray*}
The first two summands are bounded by $C_5 n^2$. It remains to bound
the last summand
\begin{eqnarray*}
&& {\mathbb E}(Q_n-T_n)^4I\bigl(Q_n-T_n>\sqrt{T_n},A_{T_n}\bigr)\\
&&\qquad\leq{\mathbb E}(Q_n-T_n)^4I\bigl(Q_n>T_n, {\mathbf
O}(T_n,Q_n-1)>c(Q_n-1-T_n)\bigr)\\
&&\qquad\leq\frac{8}{c^4}{\mathbb E}\bigl({\mathbf O}(Q_n-1) - {\mathbf
O}(T_n)\bigr)^4I(Q_n>T_n) + 8\\
&&\qquad\leq\frac{8}{c^4}{\mathbb E}\bigl({\mathbf O}(T_n)-n\bigr)^4 + 8\\
&&\qquad\leq C_6n^2,
\end{eqnarray*}
where the last inequality follows from Corollary \ref
{cor:momentmixing}. Similarly, one can show that ${\mathbb
E}(Q_n-T_n)^4I(Q_n-T_n<-\sqrt{T_n},A_{T_n}) \leq C_7 n^2$.
\end{pf*}
\begin{pf*}{Proof of Corollary \ref{cor:inversevar}}
It follows from Theorem \ref{thm:VarOn} that $a(T_n) = n + O(1)$,
$\sigma_n \asymp\sqrt n$ and $|a(m)-a(n)| \asymp|m-n|$ independently
of $m,n$. Therefore,
the first statement of Corollary \ref{cor:inversevar} follows directly
from Theorem \ref{thm:inverseCLT}. The upper bound in the second
statement follows immediately from the upper bound in the first
statement. For the lower bound, we may apply the CLT for $(a(Q_n))$ to
deduce that there exists $C>0$ such that for all $n$,
\[
\P\bigl(Q_n \geq T_n + \sqrt n\bigr) > C \quad\mbox{and}\quad \P\bigl(Q_n \leq T_n - \sqrt
n\bigr) > C.
\]
The lower bound follows from these two estimates. Indeed, if $\E Q_n
\geq T_n$, then $Q_n \leq T_n - \sqrt n$ implies that $Q_n \leq\E Q_n
- \sqrt n$ and so
\[
\E(Q_n - \E Q_n)^2 \geq n \P\bigl(Q_n \leq T_n - \sqrt n\bigr) > C n.
\]
If $\E Q_n \leq T_n$, then the argument is similar.
\end{pf*}
\begin{pf*}{Proof of Corollary \ref{cor:radii}}
The proofs of both statements are similar so we only show the proof of
the first. We first prove the lower bound.
The CLT for $(a(Q_n))$ implies that there exists $C_1$ such ${\mathbb
P}(Q_n > T_n + \sqrt n) > C_1$.
It is obvious that $\hat R_n \geq2^{Q_n-1}$. Therefore, ${\mathbb
P}(\hat R_n \geq2^{T_n + \sqrt n -1}) > C_1$, which implies that
${\mathbb E}(\log\hat R_n - T_n)^2 \geq(\sqrt n - 1)^2 C_1$.

We now prove the upper bound.
We first observe that by Theorem \ref{outletmomentthm}, using $t=4$,
\begin{eqnarray}\label{eq: rhatlowerbound}
\P\bigl(\hat R_n < 2^{\sqrt n}\bigr) &\leq&\P\bigl(Q_n \leq\sqrt n\bigr) \leq\P
\bigl({\mathbf O}\bigl(\sqrt n\bigr) \geq n\bigr) \nonumber\\[-8pt]\\[-8pt]
&\leq&\E\mathbf O\bigl(\sqrt n\bigr)^4/n^4 =
O(n^{-2}).\nonumber
\end{eqnarray}
Therefore, ${\mathbb E}(\log\hat R_n - T_n)^2I(\hat R_n < 2^{\sqrt n})
= o(n)$. We next rule out the case when $\hat R_n > 2^{C_2n}$ for large
enough $C_2$.
\begin{eqnarray*}
{\mathbb E}(\log\hat R_n - T_n)^2I(\hat R_n > 2^{C_2n})
&\leq& {\mathbb E}(\log\hat R_n)^2I(\hat R_n > 2^{C_2n}) \\
&\leq& \sum_{k=1}^\infty\bigl(C_2n(k+1)\bigr)^2{\mathbb P}(\hat R_n >
2^{C_2nk}).
\end{eqnarray*}
Note that ${\mathbb P}(\hat R_n > 2^{C_2nk})$ is bounded above by
\begin{eqnarray*}
&&{\mathbb P}\bigl(\mbox{there is no } p_c\mbox{-open circuit around the
origin in } \operatorname{Ann}\bigl(2^{C_2nk-\sqrt{C_2nk}},2^{C_2nk}\bigr)\bigr)
\\
&&\qquad{} +  {\mathbb P}\bigl(Q_n > C_2nk - \sqrt{C_2nk}\bigr).
\end{eqnarray*}
Using the RSW theorem and Lemma \ref{lem:LD} if $C_2$ is large enough,
this gives the bound
%
%
\begin{equation}\label{eq: devforR}
{\mathbb P}(\hat R_n > 2^{C_2nk}) \leq C_3\exp(-(nk)^{C_4}).
\end{equation}
Therefore,
\[
{\mathbb E}(\log\hat R_n - T_n)^2I(\hat R_n > 2^{C_2n}) = o(n).
\]
Let $\tilde A_n$ be the event that there exists a $p_c$-open
circuit around the origin in $\operatorname{Ann}(2^{k-\sqrt k},2^k)$ for all $k\geq
\sqrt n$.
It follows from the RSW theorem that
%
%
\begin{equation}\label{eq: widetildeA}
{\mathbb P}(\tilde A_n^c) \leq C_5\exp(-n^{C_6}).
\end{equation}
Therefore,
\[
{\mathbb E}(\log\hat R_n - T_n)^2I(\hat R_n \leq2^{C_2n},\tilde
A_n^c) = o(n).
\]
Moreover, if $\hat R_n > 2^{\sqrt n}$ and $\tilde A_n$ occurs, then
$Q_n \geq\log\hat R_n - \sqrt{\log\hat R_n} - 1$.
Hence
\begin{eqnarray*}
&& {\mathbb E}(\log\hat R_n - T_n)^2I\bigl(2^{\sqrt n} \leq\hat R_n \leq
2^{C_2n},\tilde A_n\bigr) \\
&&\qquad\leq 2{\mathbb E}(Q_n-T_n)^2 + 2{\mathbb E}(\log\hat R_n -
Q_n)^2I\bigl(Q_n > \log\hat R_n - \sqrt{C_2n} - 1\bigr)\\
&&\qquad\leq C_7 n.
\end{eqnarray*}
The last inequality follows from Corollary \ref{cor:inversevar} and
from the fact that $Q_n \leq\log\hat R_n +1$.
The upper bound is proved.
\end{pf*}
\begin{pf*}{Proof of Corollary \ref{cor:weights}}
For any $n$, let
\[
f(n) = \max\{\tau_e\dvtx e \mbox{ is an outlet in } B(n)^c\}
\]
and $g(n) = \min\{\tau_e\dvtx e$ is an outlet in $B(n)\}$ if there
is an outlet in $B(n)$ and $g(n)=0$ otherwise.
\begin{lemma}
There exists $C<\infty$ such that for any $t\geq1$ and $n\geq1$,
\[
{\mathbb E}\biggl(\biggl|\log\frac{f(n)-p_c}{p_n-p_c}\biggr|^t
\biggr)\leq(Ct)^{Ct}   \quad\mbox{and}\quad
{\mathbb E}\biggl(\biggl|\log\frac{g(n)-p_c}{p_n-p_c}\biggr|^t
\biggr)\leq(Ct)^{Ct}.
\]
\end{lemma}
\begin{pf}
Using the RSW theorem and (\ref{eqExpDecay}), respectively, we see that
there exist constants $C_1<\infty$ and $C_2>0$ such that for any
$n\geq1$ and $p\in(0,1)$,
\[
{\mathbb P}\bigl(f(n)<p\bigr) \leq{\mathbb P}\bigl(B(n)\stackrel{p}\leftrightarrow
\infty\bigr)
\leq
C_1\biggl(\frac{n}{L(p)}\biggr)^{C_2}
\]
and
\[
{\mathbb P}\bigl(f(n)\geq p\bigr) \leq{\mathbb P}(B_{n,p}) \leq
C_1\exp\biggl(-C_2\frac{n}{L(p)}\biggr),
\]
where $B_{n,p}$ is defined directly above (\ref{eqExpDecay}).
For $k\in{\mathbb Z}$, let $q_k = p_{2^kn}$
\begin{eqnarray*}
&& {\mathbb E}\biggl(\biggl|\log\frac{f(n)-p_c}{p_n-p_c}
\biggr|^t\biggr) \\
&&\qquad= \sum_k {\mathbb E}\biggl(\biggl|\log\frac
{f(n)-p_c}{p_n-p_c}\biggr|^t I\bigl(f(n) \in[q_{k+1},q_k)
\bigr)\biggr)\\
&&\qquad\leq \sum_{k\geq0}\biggl|\log\frac{q_{k+1}-p_c}{p_n-p_c}
\biggr|^t{\mathbb P}\bigl(f(n) < q_k\bigr)\\
&&\qquad\quad{} + \sum_{k<0}\biggl|\log\frac{q_{k}-p_c}{p_n-p_c}\biggr|^t{\mathbb
P}\bigl(f(n) \geq q_{k+1}\bigr).
\end{eqnarray*}
The first result of the lemma follows from (\ref{ineqP}) and the above
estimates.
It remains to prove the second statement.
Note that for any $n\geq1$ and $p\in(0,1)$,
\[
{\mathbb P}\bigl(g(n)<p\bigr) \leq{\mathbb P}\bigl(B(n)\stackrel{p}\leftrightarrow
\infty\bigr)
\leq
C_1\biggl(\frac{n}{L(p)}\biggr)^{C_2}.
\]
To bound ${\mathbb P}(g(n)\geq p)$, note that if $g(n)\geq p$, then
there is an outlet in~$B(n)$.
For $1 \leq m \leq\lfloor\log n\rfloor+1$, consider the event
$A_{m,n}$ that $\operatorname{Ann}(\lfloor n/2^m\rfloor, n)$ contains an outlet. (For
the case $m= \lfloor\log n\rfloor+1$, we use the convention that
$\operatorname{Ann}(\lfloor n/2^m \rfloor, n) = B(n)$.) Note that for $m=0$,
$A_{m,n}$ is equal to the null event and that for fixed $n$, the events
$A_{m,n}$ are increasing in $m$. By Lemma~\ref{lem:LD}, there exists
$C_3<\infty$ and $C_4>0$ such that for all $m,n$,
\[
{\mathbb P}(A_{m,n}^c)\leq C_3\exp(-m^{C_4}).
\]
%

Using this estimate, we get
\begin{eqnarray*}
{\mathbb P}\bigl(g(n) \geq p\bigr) &=& \sum_{m=0}^{\lfloor\log n\rfloor
}{\mathbb P}\bigl(g(n)\geq p, A_{m,n}^c, A_{m+1,n}\bigr) \\
&\leq& \sum_{m=0}^{\lfloor\log n\rfloor}{\mathbb P}\bigl(B_{\lfloor
n/2^{m+1}\rfloor,p}, A_{m,n}^c\bigr)\\
&\leq&\sum_{m=0}^{\lfloor\log n\rfloor}{\mathbb P}\bigl(B_{\lfloor
n/2^{m+1}\rfloor,p}\bigr)^{1/2}{\mathbb P}(A_{m,n}^c)^{1/2} \\
&\leq& C_5 \sum_{m=0}^{\lfloor\log n\rfloor}\biggl[\exp
\biggl(-C_2\frac{\lfloor n/2^{m+1}\rfloor}{L(p)}\biggr)\exp
(-m^{C_4})\biggr] ^{1/2}
\end{eqnarray*}
for some $C_5<\infty$. In particular, for $k<0$,
\[
{\mathbb P}\bigl(g(n) \geq q_{k}\bigr)
\leq
C_{6}\exp(-|k|^{C_7}).
\]
The remainder of the proof of the lemma is similar to the proof of the
first statement.
\end{pf}

We proceed with the proof of the corollary. We will only prove the
first statement; the proof of the second is similar. Inequality (\ref
{ineqP}) and Corollary~\ref{cor:inversevar} imply that
\[
{\mathbb E}\biggl(\log\frac{p_{2^{Q_n}}-p_c}{p_{2^{T_n}}-p_c}
\biggr)^2 \asymp{\mathbb E}(Q_n-T_n)^2 \asymp n.
\]
Note that
\[
g(2^{Q_n}) \leq\hat\tau_n \leq f(2^{Q_n-1}).
\]
Therefore the corollary will follow if we show that there exists $C_8$
such that for all~$n$,
\[
{\mathbb E}\biggl(\log\frac{g(2^{Q_n})-p_c}{p_{2^{Q_n}}-p_c}
\biggr)^2 \leq C_8\sqrt{n}   \quad\mbox{and}\quad
{\mathbb E}\biggl(\log\frac{f(2^{Q_n-1})-p_c}{p_{2^{Q_n}}-p_c}
\biggr)^2 \leq C_8\sqrt{n}.
\]
Let $D_n$ be the event that (a) there exists a $p_c$-open circuit in
the annulus $\operatorname{Ann}(2^{n-n^{1/4}},2^{n-1})$,
(b) this circuit\vspace*{1pt} is connected to infinity by a
$p_{2^{n-2n^{1/4}}}$-open path
and (c) there exists a $p_{2^{n+n^{1/4}}}$-closed dual circuit around
$B(2^n)^*$.
The RSW theorem and (\ref{eqExpDecay}) imply that there exist
constants $C_9$ and $C_{10}$ such that for all $n$,
%
%
\begin{equation}\label{eq: dnbound}
{\mathbb P}(D_n^c) \leq C_9e^{-n^{C_{10}}}.
\end{equation}
Recall that for all $n$,
\[
{\mathbb E}\biggl(\log\frac{g(2^n)-p_c}{p_{2^{n}}-p_c}\biggr)^4 \leq
C_{11}  \quad\mbox{and}\quad
{\mathbb E}\biggl(\log\frac{f(2^{n-1})-p_c}{p_{2^{n}}-p_c}\biggr)^4
\leq C_{11}.
\]
Therefore,
\[
{\mathbb E}\biggl(\log\frac{g(2^{Q_n})-p_c}{p_{2^{Q_n}}-p_c}
\biggr)^2I(D_{Q_n}^c)
\leq
\sum_{k=1}^\infty{\mathbb E}\biggl(\log\frac
{g(2^k)-p_c}{p_{2^k}-p_c}\biggr)^2I(D_k^c)
\leq C_{12},
\]
where $D_{Q_n}$ is the event $\bigcup_k ( D_k \cap\{Q_n=k\}
)$. Similarly,
\[
{\mathbb E}\biggl(\log\frac{f(2^{Q_n-1})-p_c}{p_{2^{Q_n}}-p_c}
\biggr)^2I(D_{Q_n}^c) \leq C_{12}.
\]
On the other hand, if $D_n$ occurs and, moreover, there is an outlet in
the annulus $\operatorname{Ann}(2^{n-1},2^n)$, then
\[
g(2^n) \mbox{ and } f(2^{n-1}) \mbox{ are both in }
[p_{2^{n+n^{1/4}}},p_{2^{n-2n^{1/4}}}].
\]
This observation and inequality (\ref{ineqP}) imply [note that
$\operatorname{Ann}(2^{Q_n-1},2^{Q_n})$ always contains an outlet]
\[
{\mathbb E}\biggl(\log\frac{g(2^{Q_n})-p_c}{p_{2^{Q_n}}-p_c}
\biggr)^2I(D_{Q_n})
\leq C_{13} {\mathbb E}Q_n^{1/2}
\leq C_{14} \sqrt{n}
\]
and, similarly,
\[
{\mathbb E}\biggl(\log\frac{f(2^{Q_n-1})-p_c}{p_{2^{Q_n}}-p_c}
\biggr)^2I(D_{Q_n})
\leq C_{14} \sqrt{n}.
\]
This completes the proof of the corollary.
\end{pf*}
\begin{pf*}{Proof of Corollary \ref{cor:SLLN2}}
We start with the proof of the first statement.
Take $r >1/2$. The SLLN for outlets gives that $({\mathbf
O}(n)-a(n))/n^r \to0$ a.s. as $n \to\infty$.
Theorem 1.4 in \cite{DS} states that there are constants $C_{1}>0$ and
$C_{2}<\infty$ such that with probability $1$, for all large $n$,
\[
C_1n < {\mathbf O}(n) < C_2n.
\]
This implies that there exist constants $C_{3}>0$ and $C_{4}<\infty$
such that with probability $1$, for all large $n$,
\[
C_{3}n < Q_n < C_{4}n.
\]
Therefore,
\[
\frac{{\mathbf O}(Q_n)-a(Q_n)}{n^r} \to0 \qquad\mbox{a.s. as } n \to
\infty.
\]
Because $a(Q_n)-n \asymp Q_n-T_n$, the first statement of the corollary
will follow if we show that $({\mathbf O}(Q_n)-n)/n^r \to0$ a.s.
Note that $n\leq\mathbf O(Q_n) \leq n + O_{Q_n}$ by the definition of $Q_n$.
Since\vspace*{1pt} there exists a finite constant $C_{5}$ such that (a)~$Q_n <
C_{4}n$ a.s. for all large $n$ and
(b) $\P(O_i > n^{r/2} $ for some $i=1, \ldots, C_{4}n) \leq
C_{5}/n^2$ (this second statement is a consequence of Theorem \ref
{outletmomentthm}), it follows that, a.s. for all large~$n$, $O_{Q_n}
\leq n^{r/2}$. The desired convergence follows.

The second and third statements follow easily from the first and from
estimates developed in the proofs of Corollaries \ref{cor:radii}
and \ref{cor:weights}.
Indeed, since $(Q_n-T_n)/n^r \to0$ a.s., the statements about $\log
\hat R_n$ and $\hat\tau_n$ will follow if we show that
%
%
\begin{equation}\label{eq: lasteq3}
\frac{\log\hat R_n - Q_n}{n^r} \to0  \quad\mbox{and}\quad
\frac{1}{n^r}\log\frac{\hat\tau_n - p_c}{p_{2^{Q_n}} - p_c} \to
0  \qquad\mbox{a.s.}
\end{equation}
It follows from the proof of Corollary \ref{cor:radii} and the
Borel--Cantelli lemma that there exists $C_{6}<\infty$ such that,
a.s., for all large $n$,
%
%
\begin{equation}\label{eq: lasteq1}
\log\hat R_n - \sqrt{C_{6}n} - 1 \leq Q_n \leq\log\hat R_n + 1.
\end{equation}
To see this, note first that by (\ref{eq: rhatlowerbound}), with
probability one, $\log\hat R_n \geq\sqrt n$ for all large~$n$. Next,
let $\tilde A_n$ be the event that for all $k \geq\sqrt n$, there
is a $p_c$-open circuit around the origin in the annulus
$\operatorname{Ann}(2^{k-\sqrt k},2^k)$. By (\ref{eq: widetildeA}), with probability
one the events $(\tilde A_n)$ occur for all large $n$. Last, by
(\ref{eq: devforR}) (setting $k=1$ there), there exists $C_{6}<\infty
$ such that with probability one, for all large~$n$, $\log\hat R_n
\leq C_{6}n$. Since $\tilde A_n \cap\{\sqrt n < \log\hat R_n <
C_{6}n\}$ implies (\ref{eq: lasteq1}), it in fact occurs a.s. for all
large $n$. This implies the desired SLLN for $(\log\hat R_n)$.

Similarly, one may use arguments from the proof of Corollary \ref
{cor:weights} and the Borel--Cantelli lemma to show that a.s., for all
large $n$,
%
%
\begin{equation}\label{eq: lasteq2}
p_{2^{Q_n+(C_{4}n)^{1/4}}} \leq\hat\tau_n \leq p_{2^{Q_n -
2(C_{4}n)^{1/4}}}.
\end{equation}
To prove this, define $D_n$ as in the proof of that corollary: it is
the event that (a) there exists a $p_c$-open circuit in
$\operatorname{Ann}(2^{n-n^{1/4}},2^{n-1})$, (b) this circuit is connected to infinity
by a $p_{2^{n-2n^{1/4}}}$-open path and (c) there exists a
$p_{2^{n+n^{1/4}}}$-closed dual circuit around $B(2^n)^*$. By (\ref
{eq: dnbound}), a.s. $D_n$ occurs for all large $n$. The fact that if
$D_n$ occurs and there is an outlet in $\operatorname{Ann}(2^{n-1},2^n)$,
then~$g(2^n)$ and $f(2^{n-1})$ are in the interval\vspace*{1pt}
$[p_{2^{n+n^{1/4}}},p_{2^{n-2n^{1/4}}}]$, combined\vspace*{1pt} with the fact that
$\operatorname{Ann}(2^{Q_n-1},2^{Q_n})$ always contains an outlet, shows (\ref{eq:
lasteq2}) a.s. for all large~$n$. Along with (\ref{ineqP}), this
implies the second part of (\ref{eq: lasteq3}) and completes the proof
of Corollary \ref{cor:SLLN2}.
\end{pf*}

\begin{appendix}\label{app}
\section*{Appendix: Covariance estimates}

Here we give the proof of Corollary \ref{cor:cov}. The proof we
present is directly from~\cite{Davydov}. We begin with a lemma, which
is (17.2.2) from \cite{IL}.
\begin{lemma}\label{lem: covlem1}
Suppose that $f$ is $\Sigma^k$-measurable, and $g$ is $\Sigma
_{k+m}$-measurable, and there are constants $C_1,C_2<\infty$ such that
$|f|\leq C_1$ and $|g|\leq C_2$ a.s. Then
\setcounter{equation}{0}
\begin{equation}\label{eq: coveq1}
| \E[ fg ] - \E f \E g | \leq4C_1C_2 \alpha
(m),
\end{equation}
where $\alpha(m)$ was defined in (\ref{eq: mixingdef}).
\end{lemma}
\begin{pf}
We write the left-hand side of (\ref{eq: coveq1}) as
\[
\bigl|\E\bigl[ f \E[ g-\E g  \mid \Sigma^k] \bigr] \bigr| \leq
C_1 \E\bigl[ \bigl| \E[ g - \E g  \mid \Sigma^k] \bigr|\bigr] =
C_1 \E\bigl[ f_1 \E[g-\E g \mid \Sigma^k] \bigr],
\]
where $f_1 = $ sgn$(\E[g-\E g \mid \Sigma^k])$. Since $f_1$ is $\Sigma
^k$-measurable,
\[
| \E[ fg ] - \E f \E g | \leq C_1 | \E
[ f_1g ] - \E f_1 \E g |.
\]
Similarly comparing $g$ to $g_1 = $ sgn$(\E[ g - \E g  \mid \Sigma_{k+m} ])$,
\[
| \E[ fg ] - \E f \E g | \leq C_1C_2|
\E[ f_1g_1 ] - \E f_1 \E g_1 |.
\]
Define $A = \{f_1 = 1\}$ and $B=\{g_1=1\}$. Then the right-hand side of
the above inequality is bounded above by
\begin{eqnarray*}
& & C_1C_2  | \P(A,B) + \P(A^c,B^c) - \P(A^c,B) - \P(A,B^c) \\
&&\hphantom{ C_1C_2  |}{} - \P(A)\P(B) - \P(A^c)\P(B^c) + \P(A^c)\P(B) + \P(A)\P(B^c)|,
\end{eqnarray*}
which is bounded above by $4 C_1C_2\alpha(m)$.
\end{pf}

Now we will suppose that one function is bounded and the other is in
$L^p$ for $p>1$. The following is Lemma 2.1 from \cite{Davydov}.
\begin{lemma}\label{lem: covlem2}
Suppose that $f$ is $\Sigma^k$-measurable, and $g$ is $\Sigma
_{k+m}$-measurable and that there exists $C<\infty$ such that $|g|\leq
C$ a.s. Further, suppose that there is $p>1$ such that the moment $\E
|f|^p<\infty$ exists. Then
%
%
\begin{equation}\label{eq: coveq2}
| \E[fg] - \E f \E g | \leq6C [ \E|f|^p
]^{1/p} \alpha(m)^{1/q},
\end{equation}
where $1/p + 1/q = 1$.
\end{lemma}
\begin{pf}
Let $N$ be a positive number to be chosen later and set $f_N = f I[|f|
< N]$. Applying the previous lemma to $f_N$ and $g$, we get
%
%
\begin{equation}\label{eq: coveq3}
| \E[f_Ng] - \E f_N \E g | \leq4 CN \alpha(m).
\end{equation}
To estimate the difference between this and the quantity in this lemma,
note that the left-hand side of (\ref{eq: coveq2}) is bounded above by
\[
| \E[f_Ng] - \E f_N \E g | + | \E[\tilde f_N g] -
\E\tilde f_N \E g |,
\]
where $\tilde f_N = f - f_N$. Since $|g| \leq C$, we find that the
second\vspace*{1pt} term is no bigger than $2C \E|\tilde f_N|$, and
\[
\E|\tilde f_N| = \E\bigl[ |f|^p |f|^{1-p} I[|f| \geq N] \bigr]
\leq N^{1-p} \E|f|^p.
\]
Combining this with (\ref{eq: coveq3}) gives
%
%
\begin{equation}\label{eq: coveq4}
| \E[fg] - \E f \E g | \leq4CN\alpha(m) + 2CN^{1-p} \E
|f|^p.
\end{equation}
Choosing $N = [ \E|f|^p ]^{1/p} \alpha(m)^{-1/p}$ yields
(\ref{eq: coveq2}).
\end{pf}

For the proof of Corollary \ref{cor:cov} we use a similar method to
the one given above. We let $C$ be a positive number to be chosen later
and set $g_C = g I[|g|<C]$. By Lemma \ref{lem: covlem2},
\[
| \E[f g_C] - \E f \E g_C | \leq6C [ \E|f|^p
]^{1/p} \alpha(m)^{1/p'},
\]
where $1/p + 1/p' = 1$. To estimate the difference, we write $\tilde
g_C = g-g_C$ and again see that
\[
| \E[fg] - \E f \E g | \leq| \E[fg_C] - \E f \E
g_C | + | \E[f \tilde g_C] - \E f \E\tilde g_C |.
\]
We bound the last term using H\"{o}lder's inequality by
\[
[\E|f|^p]^{1/p} [ \E|\tilde g_C - \E\tilde g_C
|^{p'} ]^{1/p'}
\]
and then use
\[
[ \E| \tilde g_C - \E\tilde g_C |^{p'}
]^{1/p'} \leq2 [\E|\tilde g_C|^{p'}]^{1/p'},
\]
which we can bound above by $2 ( C^{p'-q} \E|g|^q
)^{1/p'}$ as in (\ref{eq: coveq4}). Choosing $C=[ \E|g|^q
]^{1/q} \alpha(m)^{-1/q}$ and combining the estimates as before
completes the proof.
\end{appendix}

\section*{Acknowledgments}
We would like to thank A. Hammond for discussions related to the law of
large numbers. We also thank C. Newman and R.~van der Hofstad for
valuable comments and advice. Last we would like to acknowledge an
anonymous referee for many helpful suggestions to make the paper more
readable.


%

%
\printaddresses

\end{document}